\documentclass[12pt]{amsart}

\usepackage{amsmath,amsthm,amsfonts,amssymb}
\usepackage[all]{xy}

\numberwithin{equation}{section} 

\newtheorem{thm}[equation]{Theorem}
\newtheorem{cor}[equation]{Corollary} 
\newtheorem{lem}[equation]{Lemma} 
\newtheorem{prop}[equation]{Proposition} 
\newtheorem{example}[equation]{Example}
\theoremstyle{definition}
 
\newtheorem{remark}[equation]{Remark}

\DeclareMathOperator{\cha}{char}

\DeclareMathOperator{\id}{id}
\DeclareMathOperator{\op}{op}
\DeclareMathOperator{\ev}{ev}
\DeclareMathOperator{\cx}{cx}
\DeclareMathOperator{\rk}{rk}
\DeclareMathOperator{\coh}{H}

\DeclareMathOperator{\Ext}{Ext}
\DeclareMathOperator{\Hom}{Hom}

\DeclareMathOperator{\Jac}{Jac}

\newcommand{\ot}{\otimes}
\newcommand{\Z}{\mathbb{Z}}
\newcommand{\C}{\mathbb{C}}
\newcommand{\g}{\mathfrak{g}}
\newcommand{\D}{\mathfrak{D}}

\newcommand{\N}{\mathcal{N}}
\newcommand{\V}{\mathcal{V}}
\newcommand{\Da}{\mathcal{D}}
\newcommand{\Hecke}{\mathcal{H}}
\newcommand{\B}{\mathcal{B}}

\newcommand{\HH}{{\rm HH}}
\newcommand{\DOT}{\setlength{\unitlength}{1pt}\begin{picture}(2.5,2)
                  (1,1)\put(2,3.5){\circle*{2}}\end{picture}}
\newcommand{\YDG}{^G_G{\mathcal{YD}}}

\begin{document}


\title[Support varieties and representation type]
{\bf Support varieties and representation type of self-injective algebras}
 
\author{J\"org Feldvoss}
\address{Department of Mathematics and Statistics, University of South Alabama,
Mobile, AL 36688--0002, USA}
\email{jfeldvoss@jaguar1.usouthal.edu}
\author{Sarah Witherspoon}\thanks{The second author was partially supported by
NSF grant \#DMS-0800832 and Advanced Research Program Grant 010366-0046-2007
from the Texas Higher Education Coordinating Board.}
\address{Department of Mathematics, Texas A\&M University, College Station, TX
77843--3368, USA}
\email{sjw@math.tamu.edu}  
      
\date{May 9, 2011}
          
\maketitle


\begin{abstract}
We use the theory of varieties for modules arising from Hochschild cohomology
to give an alternative version of the wildness criterion of Bergh and Solberg
\cite{BS}: If a finite dimensional self-injective algebra has a module of
complexity at least $3$ and satisfies some finiteness assumptions on Hochschild
cohomology, then the algebra is wild. We show directly how this is related to
the analogous theory for Hopf algebras that we developed in \cite{FW}. We give
applications to many different types of algebras: Hecke algebras, reduced
universal enveloping algebras, small half-quantum groups, and Nichols (quantum
symmetric) algebras.
\end{abstract}

\medskip



\section*{Introduction}


The cohomology of a finite dimensional algebra holds information about its
representation type. Rickard first made this observation about group algebras,
and it inspired him to develop a theory for self-injective algebras (see
\cite{R}). Since then mathematicians have applied homological methods to
determine representation type in various contexts (see \cite{A,AM,EN,G2,GP,Pr})
and they have partially filled a gap in a wildness criterion of Rickard
\cite{R} (see \cite[Corrigendum]{A} and \cite{BS,F2,FW}).

For example, Farnsteiner \cite{F2} proved that if the Krull dimension of the
cohomology ring of a cocommutative Hopf algebra is at least $3$, then the Hopf
algebra has wild representation type; more precisely he gave a block version
of this statement. He used the result of Friedlander and Suslin \cite{FS}
that the cohomology ring of a finite dimensional cocommutative Hopf algebra
is finitely generated, as well as the theory of support varieties developed by
Suslin, Friedlander, Bendel, and Pevtsova \cite{FP,SFB}. In \cite{FW} we gave
a direct generalization of Farnsteiner's results to all Hopf algebras satisfying
finiteness assumptions on their cohomology. Bergh and Solberg \cite{BS} gave
a wildness criterion even more generally for self-injective algebras, under
some finiteness assumptions on the cohomology of given modules. 

In this paper, we begin by showing how a wildness criterion for self-injective
algebras may be obtained using Hochschild cohomology and the theory of support
varieties developed for Hochschild cohomology by Erdmann, Holloway, Snashall,
Solberg, and Taillefer (see \cite{EHSST} and \cite{SS}). Bergh and Solberg have
weaker hypotheses than ours (see for example \cite[Example 4.5]{BS}) due to the
flexibility their relative theory provides, however our finiteness assumptions
are known to hold for many types of algebras. Most of this article is devoted to
applications: We apply the theory to Hecke algebras, reduced universal enveloping
algebras of restricted Lie algebras, small half-quantum groups, and Nichols (quantum
symmetric) algebras. In the last section, we show that for a Hopf algebra, one may
make a choice so that the support varieties for a module defined in terms of Hochschild
cohomology herein are isomorphic to the support varieties defined in \cite{FW} in
terms of the cohomology ring of the Hopf algebra. Thus we give a direct connection
between these two seemingly different support variety theories for Hopf algebras.

Throughout this paper we will assume that all associative rings have unity elements 
and that all modules over associative rings are unital. Moreover, all modules are
left modules and all tensor products are over the ground field unless indicated
otherwise. Finally, the principal block of an augmented $k$-algebra $A$ is the
unique indecomposable ideal direct summand of $A$ whose unity element acts as the
identity on the trivial $A$-module $k$.


\section{Complexity and support varieties}\label{prelim}


Let $A$ be a finite dimensional self-injective algebra over an algebraically closed
field $k$ of arbitrary characteristic unless stated otherwise. In this section we
recall some definitions and results that will be needed later.

Let $A^e:=A\ot A^{\op}$ denote the enveloping algebra of $A$. Then the graded ring
$\Ext_{A^e}^{\DOT}(A,A)$ is isomorphic to the Hochschild cohomology ring $\HH^{\DOT}
(A)$ of $A$. For any $A$-bimodule $M$, $\Ext_{A^e}^{\DOT}(A,M)$ is isomorphic to the
Hochschild cohomology $\HH^{\DOT}(A,M)$.

We make the following assumption in order to obtain an affine variety, namely $\V_H$
defined below, in which to consider the support varieties of finite dimensional $A$-modules.

\vspace{.5cm}

\noindent Assumption ({\bf fg1}):
\vspace{-.3cm}
\begin{equation*}\begin{array}{l}
\mbox{{\em There is a graded subalgebra} } H^{\DOT}\mbox{ {\em of} }\ \HH^{\DOT}(A)
\mbox{ {\em such that} } H^{\DOT}\mbox{ {\em is a finitely}}\\
\mbox{{\em generated commutative algebra and} } H^0=\HH^0(A)\,.\end{array}
\end{equation*}
\vspace{.1cm}

For any $A$-module $M$ there is a homomorphism of graded rings $\varphi_M:\HH^{\DOT}(A)
\to\Ext_A^{\DOT}(M,M)$ given by $\varphi_M(\eta):=\eta\ot_A M$ (see \cite[p.~707]{SS}).
Combined with Yoneda composition, the homomorphisms $\varphi_M$ and $\varphi_N$ 
induce right and left $\HH^{\DOT}(A)$-actions on $\Ext_A^{\DOT}(M,N)$ for any two $A$-modules
$M$ and $N$ (see \cite[p.~707]{SS}), and by restriction they also induce right and left
$H^{\DOT}$-actions on $\Ext_A^{\DOT}(M,N)$.

Denote by $\N_H$ the ideal in $H^{\DOT}$ generated by the homogeneous nilpotent elements
of $H^{\DOT}$ and the Jacobson radical of $H^0$. Since $\N_H$ is a homogeneous ideal of
$H^{\DOT}$ and the latter is commutative (see \cite[Corollary 1, p.\ 281]{G} or \cite[Corollary~1.2]{SS}),
it follows from \cite[Lemma 2.2(b)]{SS} that every maximal ideal of $H^{\DOT}$ contains $\N_H$.
Consequently, there is a bijection between the maximal ideals of $H^{\DOT}$ and the maximal
ideals of $\overline{H^{\DOT}}:=H^{\DOT}/\N_H$. In the following we denote by $\V_H$ the
maximal ideal spectrum of $\overline{H^{\DOT}}$. 

Let now $M$ and $N$ be finite dimensional $A$-modules. Let $I_H(M,N)$ be the subset of
those elements in $\V_H$ whose inverse images in $H^{\DOT}$ contain the annihilator of
the action of $H^{\DOT}$ on $\Ext_A^{\DOT}(M,N)$. (The action may be taken to be either
the left or right action; their annihilators coincide according to \cite[Lemma 2.1]{SS}.)
Then $I_H(M,N)$ is a homogeneous ideal of $\overline{H^{\DOT}}$. Let $\V_H(M,N)$ denote
the maximal ideal spectrum of the finitely generated commutative $k$-algebra
$\overline{H^{\DOT}}/I_H(M,N)$. As the ideal $I_H(M,N)$ is homogeneous, the variety
$\V_H(M,N)$ is conical. If $M=N$, we write $I_H(M):=I_H(M,M)$ and $\V_H(M):=\V_H(M,M)$.
The latter is called the {\em support variety\/} of $M$ (see \cite[Definition 2.3]{SS}).

In order to be able to apply the results of \cite{EHSST} and \cite{SS}, we also need
the following assumption, where $\Jac(A)$ denotes the Jacobson radical of $A$.
\vspace{.5cm}

\noindent Assumption ({\bf fg2}): $\Ext_A^{\DOT}(A/\Jac(A),A/\Jac(A))$ {\em is a
finitely generated\/} $H^{\DOT}${\em -module\/}.
\vspace{.01cm}

\begin{remark}\label{TFAE}
By \cite[Proposition 2.4]{EHSST}, the assumption ({\bf fg2}) is equivalent to either
of the following two statements:
\begin{itemize}
\item[(i)]  For all finite dimensional $A$-modules $M$ and $N$, $\Ext^{\DOT}_A(M,N)$
            is finitely generated over $H^{\DOT}$.
\item[(ii)] For all finite dimensional $A$-bimodules $M$, $\HH^{\DOT}(A,M)$ is finitely
            generated over $H^{\DOT}$.
\end{itemize}
Further, the statement (i) is equivalent to the corresponding statement in which $M=N$,
since we may apply the latter to the direct sum of two finite dimensional modules
to obtain (i).
\end{remark}

In the following we show that every block of a finite dimensional self-injective algebra
for which assumptions ({\bf fg1}) and ({\bf fg2}) hold is self-injective and also satisfies
the same conditions ({\bf fg1}) and ({\bf fg2}).

\begin{lem}\label{block}
Let $A$ be a finite dimensional algebra over an arbitrary field $k$ and let $B$ be a block
of $A$. Then the following statements hold:
\begin{enumerate}
\item If $A$ is self-injective, then $B$ is self-injective.
\item If $A$ satisfies {\rm ({\bf fg1})}, then $B$ satisfies {\rm ({\bf fg1})}.
\item If $A$ satisfies {\rm ({\bf fg2})}, then $B$ satisfies {\rm ({\bf fg2})}.
\end{enumerate}
\end{lem}

\begin{proof}
(1): Since direct summands of injective modules are injective, $B$ is injective as an
$A$-module and therefore also as a $B$-module.

(2): Let $e$ be the primitive central idempotent of $A$ corresponding to $B$, and
consider the algebra homomorphism from $A^e$ onto $B^e$ defined by $a\ot a'\mapsto
ea\ot ea'$. Letting $e$ act on $f\in\Hom_k(A^{\ot n},A)$ as $(ef)(x_1\ot\cdots\ot
x_n)=e(f(ex_1\ot\cdots\ot ex_n))$, one sees that $\HH^n(B)$ is isomorphic to $e\HH^n
(A)$ where the latter is a direct summand of $\HH^n(A)$ for every non-negative integer
$n$. In particular, $eH^{\DOT}$ is a finitely generated commutative subalgebra of the
direct summand $\HH^{\DOT}(B)=e\HH^{\DOT}(A)$ of $\HH^{\DOT}(A)$. Hence $B$ satisfies
({\bf fg1}).

(3): For all $A$-modules $M$ and $N$, $\Ext_B^n(eM,eN)$ is isomorphic to $e\Ext_A^n
(M,N)$ where the latter is a direct summand of $\Ext_A^n(M,N)$. If $\xi_1,\dots,\xi_s$
are generators for the $H^{\DOT}$-module $\Ext_A^{\DOT}(M,N)$, then $e\xi_1,\dots,e\xi_s$
are generators for the $eH^{\DOT}$-module $\Ext_B^{\DOT}(eM,eN)$. Hence B satisfies
{\rm ({\bf fg2})}.
\end{proof}

Let $V_{\DOT}$ be a graded vector space over $k$ with finite dimensional homogeneous
components. The {\em rate of growth\/} $\gamma(V_{\DOT})$ of $V_{\DOT}$ is  the smallest
non-negative integer $c$ such that there is a number $b$ for which $\dim_k V_n\leq b
n ^{c-1}$ for all non-negative integers $n$. If no such $c$ exists, $\gamma(V_{\DOT})$
is defined to be $\infty$. The {\em complexity\/} $\cx_A(M)$ of an $A$-module $M$ may
be defined in the standard way: Let $P_{\DOT}: \ \cdots\rightarrow P_1\rightarrow P_0
\rightarrow M\rightarrow 0$ be a minimal projective resolution of $M$. Then $\cx_A(M)
:=\gamma(P_{\DOT})$.

We will need the following lemma stating that for a module belonging to a block of a
finite dimensional self-injective algebra, the complexity is the same whether it is
considered as a module for the block or for the whole algebra.

\begin{lem}\label{cxblock}
Let $A$ be a finite dimensional self-injective algebra over an arbitrary field and
let $B$ be a block of $A$. Then $\cx_B(M)=\cx_A(M)$ for every $B$-module $M$.
\end{lem}

\begin{proof}
Every projective $A$-module is projective as a $B$-module, and consequently $\cx_B(M)
\leq\cx_A(M)$. On the other hand, $B$ is a direct summand of $A$, and as such $B$ is
a projective $A$-module. Hence every projective $B$-module is also projective as an
$A$-module and therefore $\cx_B(M)\geq\cx_A(M)$.
\end{proof}

Finally, we will need the following connection between complexity, support varieties,
and the rate of growth of the self-extensions of a module. The first equality is just
\cite[Theorem 2.5(c)]{EHSST}. Note that the indecomposability of $A$ is not used in
the proof (see the proof of \cite[Proposition 2.1(a)]{EHSST}). The second equality is
implicitly contained in the suggested proof (see also the first part of the proof of
\cite[Proposition 2.3]{FW} for Hopf algebras that transfers to this more general
situation under the assumptions {\rm ({\bf fg1})} and {\rm ({\bf fg2})}).

\begin{thm}\label{cx}
Let $A$ be a finite dimensional self-injective algebra over an algebraically closed
field $k$ of arbitrary characteristic for which assumptions {\rm ({\bf fg1})} and
{\rm ({\bf fg2})} hold. Let $M$ be a finite dimensional $A$-module. Then
$$
\cx_A(M)=\dim\V_H(M)=\gamma(\Ext^{\DOT}_A(M,M))\,.
$$
\end{thm}


\section{Complexity and representation type}


Theorem \ref{mainthm} below generalizes \cite[Theorem 3.1]{F2} and \cite[Theorem 3.1]{FW}
(cf.\ \cite[Theorem 4.1]{BS}). Thus we recover a result of Rickard \cite[Theorem 2]{R},
whose proof contains a gap (see \cite[Corrigendum]{A} and \cite{F2}), under the assumptions
({\bf fg1}) and ({\bf fg2}). We apply the general theory developed in \cite{EHSST} and
\cite{SS} (the needed parts of which are summarized in Section \ref{prelim}) to
adapt Farnsteiner's proof to this more general setting. Theorem \ref{mainthm} is
similar to \cite[Theorem 4.1]{BS} with slightly different hypotheses; for completeness
we give a proof in our context.

Recall that there are three classes of finite dimensional associative algebras over
an algebraically closed field (see \cite[Corollary C]{CB} or \cite[Theorem 4.4.2]{B1}):
An algebra $A$ is {\em representation-finite\/} if there are only finitely many
isomorphism classes of finite dimensional indecomposable $A$-modules. It is {\em tame\/}
if it is not representation-finite and if the isomorphism classes of indecomposable
$A$-modules in any fixed dimension are almost all contained in a finite number of
$1$-parameter families. The algebra $A$ is {\em wild\/} if the category of finite
dimensional $A$-modules contains the category of finite dimensional modules over the
free associative algebra in two indeterminates. (For more precise definitions we refer
the reader to \cite[Definition~4.4.1]{B1}.) The classification of indecomposable objects
(up to isomorphism) of the latter category is a well-known unsolvable problem and so
one is only able to classify the finite dimensional indecomposable modules of
representation-finite or tame algebras.

\begin{thm}\label{mainthm}
Let $A$ be a finite dimensional self-injective algebra over an algebraically closed
field $k$ of arbitrary characteristic for which assumptions {\rm ({\bf fg1})} and
{\rm ({\bf fg2})} hold and let $B$ be a block of $A$. If there is a $B$-module $M$
such that $\cx_A(M)\ge 3$, then $B$ is wild.
\end{thm}

\begin{proof}
Let $M$ be a $B$-module for which $\cx_A(M)\geq 3$. It follows from Lemma ~\ref{cxblock}
that $\cx_B(M)=\cx_A(M)\geq 3$. According to Lemma \ref{block}, $B$ is also a finite
dimensional self-injective algebra over $k$ for which assumptions {\rm ({\bf fg1})}
and {\rm ({\bf fg2})} hold. As a consequence, it is enough to prove the assertion
for $A$ and one can assume that $A=B$.

It follows from Theorem \ref{cx} and our hypotheses that $n:=\dim\V_H(M)=\cx_A(M)\ge 3$.
For each $\zeta\in\overline{H^{\DOT}}$, denote by $\langle\zeta\rangle$ the ideal of
$\overline{H^{\DOT}}$ generated by $\zeta$. For each ideal $I$ of $\overline{H^{\DOT}}$,
denote by $Z(I)$ its zero set, that is the set of all maximal ideals of $\overline{H^{\DOT}}$
containing $I$. Since $\V_H(M)$ is a conical subvariety of $\V_H$, an application of the
Noether Normalization Lemma (see \cite[Lemma 1.1]{F2}) yields non-zero elements $\eta_s$
($s\in k$) in $\overline{H^d}$ for some $d>0$ such that
\begin{itemize}
\item[(i)]  $\dim Z(\langle\eta_s\rangle)\cap\V_H(M)=n-1$ for all $s\in k$,
\item[(ii)] $\dim Z(\langle\eta_s\rangle)\cap Z(\langle\eta_t\rangle)\cap\V_H(M)
            =n-2$ for $s\neq t\in k$.
\end{itemize}
According to \cite[Proposition 4.3 and Definition 3.1]{EHSST}, there are $A$-bimodules
$M_{\eta_s}$ for which the corresponding $A$-modules $N_{\eta_s}:=M_{\eta_s}\ot_AM$
satisfy
$$
\V_H(N_{\eta_s})=Z(\langle\eta_s\rangle)\cap\V_H(M)
    \ \ \mbox{ and } \ \
\dim_kN_{\eta_s}\leq b_A\,,
$$
where $b_A:=(\dim_kM)(\dim_kA)(\dim_k\Omega_{A^e}^{d-1}(A))$, for all $s\in k$.
Decompose $N_{\eta_s}$ into a direct sum of indecomposable $A$-modules. By
\cite[Proposition 3.4(f)]{SS} and (i) above, for each $s$ there is an indecomposable
direct summand $X_s$ of $N_{\eta_s}$ such that
$$
\V_H(X_s)\subseteq\V_H(N_{\eta_s})=Z(\langle\eta_s\rangle)\cap\V_H(M)
$$
and $\dim\V_H(X_s)=n-1$. We claim that $\V_H(X_s)\neq\V_H(X_t)$ when $s\neq t$: Note
that if $\V_H(X_s)=\V_H(X_t)$, then
$$
\V_H(X_s)\subseteq Z(\langle\eta_s\rangle)\cap Z(\langle\eta_t\rangle)\cap\V_H(M)\,,
$$
and the latter variety has dimension $n-2$ if $s\neq t$. It follows that the varieties
$\V_H(X_s)$ are distinct for different values of $s$, implying that the indecomposable
$A$-modules $X_s$ are pairwise non-isomorphic. The dimensions of the modules $X_s$ are
all bounded by $b_A$, since $X_s$ is a direct summand of $N_{\eta_s}$. Consequently,
there are infinitely many non-isomorphic indecomposable $A$-modules $X_s$ of some fixed
dimension. According to Theorem \ref{cx}, $\cx_A(X_s)=\dim\V_H(X_s)=n-1\geq 2$.

If $A$ is not wild, then by \cite[Corollary C]{CB}, $A$ is tame or representation-finite.
It follows from \cite[Theorem D]{CB} that only finitely many indecomposable $A$-modules
of any dimension (up to isomorphism) are not isomorphic to their Auslander-Reiten
translates. Since $A$ is self-injective, the Auslander-Reiten translation $\tau_A$ is
the same as $\N_A\circ\Omega_A^2=\Omega_A^2\circ\N_A$ where $\N_A$ denotes the
Nakayama functor of $A$ (see \cite[Proposition IV.3.7(a)]{ARS}). The Nakayama functor
preserves projective modules and so any $A$-module isomorphic to its Auslander-Reiten
translate is periodic and thus has complexity $1$. Hence in any dimension there are only
finitely many isomorphism classes of indecomposable $A$-modules with complexity not
equal to~$1$. This is a contradiction, since we have shown that for some dimension, there
are infinitely many non-isomorphic indecomposable $A$-modules of complexity greater
than $1$. Therefore $A$ is wild.
\end{proof}

An application of Theorem \ref{mainthm} and the Trichotomy Theorem \cite[Corollary C]{CB} is:

\begin{cor}\label{maincor}
Let $A$ be a finite dimensional self-injective algebra over an algebraically closed field
$k$ of arbitrary characteristic for which assumptions {\rm ({\bf fg1})} and {\rm ({\bf fg2})}
hold and let $B$ be a block of $A$. If $B$ is tame, then $\cx_A(M)\le 2$ for every finite
dimensional $B$-module $M$.
\end{cor}

\begin{remark}
Note that the tameness of an algebra $A$ implies the tameness of every block of $A$. In
particular, Corollary \ref{maincor} holds also for the whole algebra.
\end{remark}


\section{Hecke algebras}


Let $\Hecke_q$ be the Hecke algebra associated to a finite Coxeter group of classical
type over a field $k$, with non-zero parameter $q$ in $k$. It is well-known that
$\Hecke_q$ is self-injective (see \cite[Theorem 2.3]{DJ1} for type A and \cite[Remark
after Theorem 2.8]{DJ1} for types B and D). Let $H^{\DOT}:=\HH^{\ev}(\Hecke_q):=\bigoplus_{n
\geq 0}\HH^{2n}(\Hecke_q)$, the subalgebra of $\HH^{\DOT}(\Hecke_q)$ generated by elements
of even degree. It follows from a recent result of Linckelmann \cite[Theorem 1.1]{Li2},
Remark \ref{TFAE}, and \cite[Corollary 1, p.\ 281]{G} that assumptions {\rm ({\bf fg1})}
and {\rm ({\bf fg2})} hold for $\Hecke_q$ as long as the characteristic of the ground
field is zero and the order of $q$ is odd if $\Hecke_q$ is of type B or D, as stated
in the next theorem. To obtain this statement for the even part of the Hochschild
cohomology from Linckelmann's result, first note that we may take as generators of the
full Hochschild cohomology ring a finite set of homogeneous elements. If we take those
generators of even degree combined with all products of pairs of generators in odd degree,
we obtain a finite set of generators of the even subring, since squares of odd degree
elements are zero. A similar argument applies to modules over the Hochschild cohomology
ring.

\begin{thm}\label{cohheckealg}
Assume that the characteristic of $k$ is zero. Let $q$ be a primitive $\ell$th root of
unity in $k$, for some integer $\ell>1$, and assume that $\ell$ is an odd integer if
$\Hecke_q$ is of type {\rm B} or {\rm D}. Then the even Hochschild cohomology ring
$\HH^{\ev}(\Hecke_q)$ is a finitely generated commutative algebra. Moreover,
$\Ext_{\Hecke_q}^{\DOT}(M,N)$ is finitely generated as an $\HH^{\ev}(\Hecke_q)$-module
for all finite dimensional $\Hecke_q$-modules $M$ and $N$.
\end{thm}

A crucial tool in finding the representation types of blocks of Hecke algebras by
applying Theorem \ref{mainthm} is the following lemma.

\begin{lem}\label{tens}
If $A$ and $A^\prime$ are finite dimensional self-injective augmented algebras over
an arbitrary field $k$ for which assumptions {\rm ({\bf fg1})} and {\rm ({\bf fg2})}
hold, then $\cx_{A\ot A^\prime}(k)=\cx_A(k)+\cx_{A^\prime}(k)$.
\end{lem}

\begin{proof}
By the K\"unneth Theorem, the tensor product of projective resolutions of $k$ as
an $A$-module and as an $A'$-module is a projective resolution of $k$ as an $A
\ot A'$-module, and moreover, $\Ext^{\DOT}_{A\ot A'}(k,k)\cong\Ext^{\DOT}_A(k,k)
\ot\Ext^{\DOT}_{A'}(k,k)$. It is well-known that the Hochschild cohomology of a
tensor product of algebras is the super tensor product of the Hochschild cohomology
of the factors (see e.g.\ \cite[Theorem 4.2.5]{L} for the analog for Hochschild homology).
As {\rm ({\bf fg1})} and {\rm ({\bf fg2})} hold for $A$ and $A^\prime$ by hypothesis,
{\rm ({\bf fg1})} and {\rm ({\bf fg2})} also hold for $A\ot A'$. Hence the result follows
from Theorem \ref{cx}. (Note that the proof of Theorem \ref{cx} does not use the
self-injectivity of the algebra.)
\end{proof}

Let $\Hecke_q(\mathrm{A}_r)$ denote the Hecke algebra of type A$_r$ at a primitive
root of unity $q$ of order $\ell$ where $2\le\ell\le r$. Erdmann and Nakano \cite{EN}
investigated the representation types of the blocks of $\Hecke_q(\mathrm{A}_r)$ over
an arbitrary field. (Note that according to \cite[Theorem 4.3]{DJ2}, $\Hecke_q
(\mathrm{A}_r)$ is semisimple unless $2\le\ell\le r$.) The simple $\Hecke_q
(\mathrm{A}_r)$-modules are in bijection with the $\ell$-regular partitions of $r$
(see \cite[Theorem 7.6]{DJ1}). Let $D^\lambda$ denote the simple $\Hecke_q
(\mathrm{A}_r)$-module corresponding to the $\ell$-regular partition $\lambda$. Then
$D^\lambda$ and $D^\mu$ belong to the same block of $\Hecke_q(\mathrm{A}_r)$ if and
only if $\lambda$ and $\mu$ have the same $\ell$-core (see \cite[Theorem 4.13]{DJ2}).
The $\ell$-core of a partition $\lambda$ of $r$ is the partition whose Young diagram
is obtained from the Young diagram of $\lambda$ by removing as many rim $\ell$-hooks
as possible. Let $\B_\lambda$ denote the block of $\Hecke_q(\mathrm{A}_r)$ that contains
the simple module $D^\lambda$. The weight $w(\lambda)$ of the block $\B_\lambda$ is
defined by $\vert\gamma\vert+\ell\cdot w(\lambda)=r$, where $\gamma$ is the $\ell$-core
of $\lambda$.

Erdmann and Nakano \cite[Proposition 3.3(A)]{EN} gave a proof of the following result
for fields of arbitrary characteristic by applying Rickard's wildness criterion
\cite[Theorem 2]{R} whose proof contains a gap. By using our Theorem \ref{mainthm}
and Linckelmann's result (Theorem \ref{cohheckealg}) we can recover the result of
Erdmann and Nakano for fields of characteristic zero.

\begin{thm}\label{heckealgtypA}
Let $\B_\lambda$ be a block of the Hecke algebra $\Hecke_q(\mathrm{A}_r)$ over an
algebraically closed field $k$ of characteristic zero at a primitive root of unity
$q$ in $k$ of order $\ell$, where $2\le\ell\le r$. If $w(\lambda)\ge 3$, then
$\B_\lambda$ is wild.
\end{thm}

\begin{proof}
Let $Y^\lambda$ denote the $q$-Young module that is the unique indecomposable
direct summand of the $q$-permutation module of the partition $\lambda$ containing
$D^\lambda$ (see \cite[Section 2]{DD}). Let $\Hecke_q(\rho)$ denote the $q$-Young
vertex of $Y^\lambda$. Erdmann and Nakano \cite[Theorem 2.2]{EN} proved that
$\cx_{\Hecke_q(\mathrm{A}_r)}(Y^\lambda)=\cx_{\Hecke_q(\rho)}(k)$. As $w(\lambda)
\ge 3$, $\Hecke_q(\rho)$ is a free module over the subalgebra $\Hecke_q
(\mathrm{A}_\ell)\ot\Hecke_q(\mathrm{A}_\ell)\ot\Hecke_q(\mathrm{A}_\ell)$ (see
\cite[proof of Proposition 3.3(A)]{EN}). By virtue of \cite[Theorem 4.3]{DJ2},
$\Hecke_q(\mathrm{A}_\ell)$ is not semisimple and therefore $\cx_{\Hecke_q
(\mathrm{A}_\ell)}(k)\ge 1$. It follows from the above and Lemma \ref{tens} that
$$
\cx_{\Hecke_q(\mathrm{A}_r)}(Y^\lambda)=\cx_{\Hecke_q(\rho)}(k)\ge\cx_{\Hecke_q
(\mathrm{A}_\ell)\ot\Hecke_q(\mathrm{A}_\ell)\ot\Hecke_q(\mathrm{A}_\ell)}(k)=3
\cdot\cx_{\Hecke_q(\mathrm{A}_\ell)}(k)\ge 3\,.
$$
Now the $q$-Young module $Y^\lambda$ belongs to $\B_\lambda$ by definition, and
thus Theorem \ref{mainthm} in conjunction with Theorem \ref{cohheckealg} yields
the assertion.
\end{proof}

\begin{remark}\label{princblocheckealgtypA}
For the principal block of the Hecke algebra $\Hecke_q(\mathrm{A}_r)$, the conclusion
of Theorem \ref{heckealgtypA} is also an immediate consequence of Theorem \ref{cx},
\cite[Theorem 1.1]{BEM}, and Theorem \ref{mainthm}.
\end{remark}

Linckelmann's result \cite[Theorem 1.1]{Li2} neither covers the two-parameter
Hecke algebras nor the one-parameter Hecke algebras of types B and D at parameters
of even order. Ariki applies Rickard's wildness criterion in his paper only in
the two-parameter case or for parameters of degree two (see \cite[Corrigendum]{A}),
so Theorem \ref{mainthm} cannot be used to give an alternative approach to fix the
gap as for type A. On the other hand, as in Remark \ref{princblocheckealgtypA}, we
obtain from Theorem \ref{cx} and \cite[Theorem 6.2, Theorem 6.6, and Theorem 1.1]{BEM}
in conjunction with Theorem \ref{mainthm} the following result.

\begin{thm}\label{heckealgtypBD}
Let $\Hecke_q$ be the Hecke algebra of type {\rm B}$_r$ or {\rm D}$_r$ over an
algebraically closed field $k$ of characteristic zero at a primitive $\ell$th
root $q$ of unity in $k$ for some odd integer $\ell>1$. If $r\ge 3\ell$, then
the principal block of $\Hecke_q$ is wild.
\end{thm}

Let now $\Hecke_q$ be the Hecke algebra of type B$_r$ or D$_r$ over an arbitrary
field $k$ of characteristic zero at a primitive $\ell$th root $q$ of unity in $k$
for some odd integer $\ell>1$. Then it follows from \cite[Theorem 4.17]{DJ3} (see
also \cite[Theorem 38]{A}) for type B$_r$, \cite[(3.6) and (3.7)]{P} for type D$_r$
($r$ odd), and \cite[Main result, p.\ 410]{H} for type D$_r$ ($r$ even) in conjunction
with Uno's result for type A (see \cite[(1.1)]{AM}) that $\Hecke_q$ is representation-finite
if and only if $r<2\ell$ (see also \cite[pp.\ 135/136]{AM} for the two-parameter
Hecke algebra). Consequently, Theorem \ref{heckealgtypBD} reduces the proof of
\cite[Theorem 57(1)]{A} to the cases $2\ell\le r<3\ell$ in which, according to
Linckelmann's result \cite[Theorem 1.2]{Li2} and Theorem \ref{cx}, all $\Hecke_q$-modules
have complexity at most $2$ and therefore a wildness criterion \`a la Rickard would
never apply.

On the other hand, a similar argument as above can be used to give a short proof
for the wildness of the Hecke algebra $\Hecke_q$ of type B$_r$ or D$_r$ over an
arbitrary field $k$ of characteristic $\neq 2$ at a primitive $\ell$th root $q$
of unity in $k$ for some odd integer $\ell>1$ by applying \cite[Proposition
3.3(B)]{EN}. According to the latter, $\Hecke_q(\mathrm{A}_{2\ell})$ is wild.
Then \cite[Theorem 4.17]{DJ3} (or \cite[Theorem 38]{A}) for type B and \cite[Main
result, p.\ 410]{H} for type D imply that $\Hecke_q(\mathrm{B}_{2\ell})$ and
$\Hecke_q(\mathrm{D}_{2\ell})$ are wild as they contain $\Hecke_q(\mathrm{A}_{2\ell})$
as an ideal direct summand. Now \cite[Corollary 4(2)]{A} shows that $\Hecke_q
(\mathrm{B}_r)$ and $\Hecke_q(\mathrm{D}_r)$ are wild as long as $n\ge 2\ell$.


\section{Reduced universal enveloping algebras}\label{ruea}


Let $\g$ be a finite dimensional restricted Lie algebra over a field $k$ of
prime characteristic, let $\chi$ be any linear form on $\g$, and let $u(\g,
\chi)$ denote the $\chi$-reduced universal enveloping algebra of $\g$ (see
\cite[Section 5.3]{SF}). Note that $u(\g,\chi)$ is always a Frobenius algebra
(see \cite[Corollary 5.4.3]{SF}), and therefore $u(\g,\chi)$ is self-injective
but it is a Hopf algebra only if $\chi=0$. In this section we give a unified
proof of the wildness criterion for $\chi$-reduced universal enveloping
algebras. In \cite[Theorem~4.1 and Corollary 4.2]{F2} this was done by
repeating parts of the proof for $\chi=0$ in the general case. We need the
following result assuring that assumptions {\rm ({\bf fg1})} and {\rm ({\bf fg2})}
hold for finite dimensional $\chi$-reduced universal enveloping algebras.

\begin{lem}\label{fgredenvalg}
Let $\g$ be a finite dimensional restricted Lie algebra over a field $k$ of prime
characteristic and let $\chi$ be any linear form on $\g$. Then there is a finitely
generated commutative graded subalgebra $H^{\DOT}$ of $\HH^{\DOT}(u(\g,\chi))$ such
that $H^0=\HH^0(u(\g,\chi))$ and $\Ext^{\DOT}_{u(\g,\chi)}(M,N)$ is finitely generated
as an $H^{\DOT}$-module for all finite dimensional $u(\g,\chi)$-modules $M$ and $N$.
\end{lem}

\begin{proof}
According to \cite[Theorem 2.4]{F1}, there exists a natural isomorphism
$$
\HH^{\DOT}(u(\g,\chi))\cong\Ext^{\DOT}_{u(\g,\chi)^e}(u(\g,\chi),u(\g,\chi))
\cong\Ext^{\DOT}_{u(\g,0)}(k,u(\g,\chi))\,,
$$
where $u(\g,\chi)$ is a restricted $\g$-module (or equivalently, $u(\g,\chi)$ is
a $u(\g,0)$-module) via $x\cdot u:=xu-ux$ for every $x\in\g$ and every $u\in u(\g,
\chi)$. It is clear that the unit map from $k$ to $u(\g,\chi)$ is a homomorphism
of restricted $\g$-modules (i.e., a $u(\g,0)$-module homomorphism) and therefore
there exists a natural homomorphism from $\Ext^{\DOT}_{u(\g,0)}(k,k)$ into
$\Ext^{\DOT}_{u(\g,0)}(k,u(\g,\chi))\cong\HH^{\DOT}(u(\g,\chi))$. In particular,
by virtue of \cite[Theorem 1.1]{FS}, the factor algebra $E^{\DOT}$ of
$\Ext^{\DOT}_{u(\g,0)}(k,k)$ modulo the kernel of this homomorphism is a finitely
generated graded subalgebra of $\HH^{\DOT}(u(\g,\chi))$. Now set
\begin{eqnarray*}
H^{\DOT}:=
\left\{
\begin{array}{cl}
\HH^0(u(\g,\chi))\cdot\bigoplus\limits_{n=0}^\infty E^n\,, & \mbox{if }\cha k=2\,,\\ & \\
\HH^0(u(\g,\chi))\cdot\bigoplus\limits_{n=0}^\infty E^{2n}\,, & \mbox{if }\cha k\neq 2\,.
\end{array}
\right.
\end{eqnarray*}
Then $H^{\DOT}$ is a finitely generated commutative graded subalgebra of $\HH^{\DOT}
(u(\g,\chi))$ such that $H^0=\HH^0(u(\g,\chi))$. Finally, the second part of the
assertion follows from \cite[Corollary 2.5]{F1} and another application of \cite[Theorem
1.1]{FS}.
\end{proof}

Theorem \ref{mainthm} and Corollary \ref{maincor} in conjunction with \cite[Corollary
5.4.3]{SF} and Lemma \ref{fgredenvalg} imply \cite[Theorem 4.1]{F2} and \cite[Corollary
4.2]{F2}, respectively. We give precise statements here for completeness. 

\begin{thm}\label{mainthmredenvalg}
Let $\g$ be a finite dimensional restricted Lie algebra over an algebraically
closed field of prime characteristic, let $\chi$ be any linear form on $\g$,
and let $B$ be any block of the $\chi$-reduced universal enveloping algebra
$u(\g,\chi)$ of $\g$. If there is a $B$-module $M$ such that $\cx_{u(\g,\chi)}
(M)\ge 3$, then $B$ is wild.
\end{thm}

\begin{cor}\label{maincorredenvalg}
Let $\g$ be a finite dimensional restricted Lie algebra over an algebraically
closed field of prime characteristic, let $\chi$ be any linear form on $\g$,
and let $B$ be any block of the $\chi$-reduced universal enveloping algebra
$u(\g,\chi)$ of $\g$. If $B$ is tame, then $\cx_{u(\g,\chi)}(M)\le 2$ for every
$B$-module $M$.
\end{cor}


\section{Small half-quantum groups}\label{small}


In this section, we give a wildness criterion for some small half-quantum groups,
that is those corresponding to certain nilpotent subalgebras of complex simple
Lie algebras. These are not Hopf algebras themselves, and so our main results
in \cite{FW} do not apply directly. 

Let $\g$ be a finite dimensional complex simple Lie algebra, $\Phi$ its root
system, and $r:=\rk(\Phi)$ its rank. Let $\ell>1$ be an odd integer and assume
that $\ell$ is not divisible by $3$ if $\Phi$ is of type G$_2$. Let $q$ be a
primitive complex $\ell$th root of unity and let $u_q(\g)$ denote Lusztig's
small quantum group \cite{Lu}.

Fix a set of simple roots and let $\Phi^+$ and $\Phi^-$ be the corresponding
sets of positive and negative roots, respectively. Then $\g$ has a standard
Borel subalgebra corresponding to $\Phi^+$ and an opposite standard Borel
subalgebra corresponding to $\Phi^-$. Let $u_q^{\ge 0}(\g)$ (denoted $u_q^+(\g)$
in \cite{FW}) be the Hopf subalgebra of $u_q({\g})$ corresponding to the standard
Borel subalgebra of $\g$ and let $u_q^{\le 0}(\g)$ denote the Hopf subalgebra
of $u_q(\g)$ corresponding to the opposite standard Borel subalgebra of $\g$.
Furthermore, we will use the notation $u_q^{>0}(\g)$ to denote the subalgebra
of $u_q({\g})$ corresponding to the largest nilpotent ideal of the standard
Borel subalgebra of $\g$ and $u_q^{<0}(\g)$ to denote the subalgebra of $u_q
(\g)$ corresponding to the largest nilpotent ideal of the opposite standard
Borel subalgebra of $\g$. Note that $u_q^{\le 0}(\g)\cong u_{q^{-1}}^{\ge 0}
(\g)$ and $u_q^{<0}(\g)\cong u_{q^{-1}}^{>0}(\g)$ (see \cite[(1.2.10)]{CK}).

Drupieski \cite[Theorem 5.6]{D} proved the following result for the full cohomology
ring. In the usual way (see the argument before Theorem \ref{cohheckealg}) one
obtains from this the corresponding result for the even degree cohomology ring.
In fact, Drupieski more generally deals with Frobenius-Lusztig kernels of quantum
groups; see \cite[p.\ 58]{D} for his notation in the case of a small quantum group. 

\begin{thm}\label{cohuqn}
Let $q$ be a primitive complex $\ell$th root of unity and assume that $\ell>1$
is an odd integer not divisible by $3$ if $\Phi$ is of type {\rm G}$_2$. Then
the even cohomology ring $\coh^{\ev}(u_q^{>0}(\g),\C):=\bigoplus_{n\ge 0}
\coh^{2n}(u_q^{>0}(\g),\C)$ is finitely generated. Moreover, if $M$ is a finite
dimensional $u_q^{>0}(\g)$-module, then $\coh^{\DOT}(u_q^{>0}(\g),M)$ is
finitely generated as an $\coh^{\ev}(u_q^{>0}(\g),\C)$-module.
\end{thm}

Let $G:=(\Z/\ell\Z)^r$, a subgroup of the group of units of $u_q^{\ge 0}(\g)$. Then
$u_q^{\geq 0}(\g)$ is a skew group algebra formed from its subalgebra $u^{>0}(\g)$
and the action of $G$ by conjugation. Consequently $G$ acts on the cohomology of
$u_q^{>0}(\g)$. Let $H^{\DOT}:=\HH^0(u_q^{>0}(\g))\cdot\coh^{\ev}(u_q^{>0}(\g),\C)^G
=\HH^0(u_q^{>0}(\g))\cdot\Ext^{\ev}_{u_q^{>0}(\g)}(\C,\C)^G$. By Theorem \ref{cohuqn},
$H^{\DOT}$ is finitely generated. We show next that it embeds into $\HH^{\DOT}
(u_q^{>0}(\g))$; for details in a more general context, see the next section. By standard
arguments, there is an isomorphism of Hochschild cohomology, $\HH^{\DOT}(u_q^{\geq
0}(\g))\cong\HH^{\DOT}(u_q^{>0}(\g), u_q^{\geq 0}(\g))^G$. Similarly there is an isomorphism
$\Ext^{\DOT}_{u^{\geq 0}_q(\g)}(\C,\C)\cong\Ext^{\DOT}_{u_q^{>0}(\g)}(\C,\C)^G$. Since
$u_q^{\geq 0}(\g)$ is a Hopf algebra, there is an embedding of $\Ext^{\DOT}_{u_q^{\geq 0}
(\g)}(\C,\C)$ into its Hochschild cohomology $\HH^{\DOT}(u_q^{\geq 0}(\g))$ (see the last
section for details and references). It can be checked that the image of $\Ext^{\DOT}_{u_q^{>0}
(\g)}(\C,\C)^G$ in $\HH^{\DOT}(u_q^{>0}(\g), u_q^{\geq 0}(\g))^G$, under these isomorphisms,
is contained in $\HH^{\DOT}(u_q^{>0}(\g))^G$. Therefore ({\bf fg1}) holds for $u_q^{>0}(\g)$.
We now claim that by Theorem \ref{cohuqn} and Remark \ref{TFAE}, ({\bf fg2}) holds for
$u_q^{>0}(\g)$: We must show that for all finite dimensional $u_q^{>0}(\g)$-modules $M$
and $N$, $\Ext^{\DOT}_{u_q^{>0}(\g)}(M,N)$ is finitely generated as an $H^{\DOT}$-module.
Since $N$ is a $u_q^{>0}(\g)$-direct summand of the module $u_q^{\ge 0}(\g)\ot_{u_q^{>0}
(\g)} N$ induced to $u_q^{\ge 0}(\g)$ and restricted back to $u_q^{>0}(\g)$, this will be true
if it is true for $N$ replaced by this induced module. By the Eckmann-Shapiro Lemma we have
$$
\Ext^{\DOT}_{u_q^{>0}(\g)}(M,u_q^{\ge 0}(\g)\ot_{u_q^{>0}(\g)}N)\cong
\Ext^{\DOT}_{u_q^{\ge 0}(\g)}(u_q^{\ge 0}(\g)\ot_{u_q^{>0}(\g)}M,u_q^{\ge 0}
(\g)\ot_{u_q^{>0}(\g)}N)\,.
$$
The latter is finitely generated over $\Ext^{\ev}_{u_q^{\ge 0}(\g)}(\C,\C)\cong
\Ext^{\ev}_{u_q^{>0}(\g)}(\C,\C)^G$, 
which embeds into $H^{\DOT}$. Since the actions correspond under these isomorphisms,
({\bf fg2}) does indeed hold for $u_q^{>0}(\g)$. For ease of reference we summarize
this in the next result.

\begin{cor}\label{fguqn}
Let $q$ be a primitive complex $\ell$th root of unity and assume that $\ell>1$ is
an odd integer not divisible by $3$ if $\Phi$ is of type {\rm G}$_2$. Then the
algebra $H^{\DOT}:=\HH^0(u_q^{>0}(\g))\cdot\coh^{\ev}(u_q^{>0}(\g),\C)^{(\Z/
\ell\Z)^r}$ is finitely generated. Moreover, $\Ext^{\DOT}_{u_q^{>0}(\g)}(M,N)$ is
finitely generated as an $H^{\DOT}$-module for all finite dimensional $u_q^{>0}
(\g)$-modules $M$ and $N$.
\end{cor}

In order to be able to apply Theorem \ref{mainthm}, we will need the following
relation between the complexities of the trivial modules for $u_q^{>0}(\g)$ and
for $u_q^{\ge 0}(\g)$.

\begin{lem}\label{cxeq}
Let $q$ be a primitive complex $\ell$th root of unity and assume that $\ell>1$
is an odd integer not divisible by $3$ if $\Phi$ is of type {\rm G}$_2$. Then
$\cx_{u_q^{>0}(\g)}(\C)=\cx_{u_q^{\ge 0}(\g)}(\C)$.
\end{lem}

\begin{proof} Since $u_q^{\ge 0}(\g)$ is a free $u_q^{>0}(\g)$-module, every
projective resolution of $u_q^{\ge 0}(\g)$-modules restricts to a projective
resolution of $u_q^{>0}(\g)$-modules. Consequently, we have that $\cx_{u_q^{>0}
(\g)}(\C)\le\cx_{u_q^{\ge 0}(\g)}(\C)$.

On the other hand, it follows from $u_q^{\ge 0}(\g)\cong\left(\Z/\ell\Z\right)^r
\cdot u_q^{>0}(\g)$, the Hochschild-Serre spectral sequence, and Maschke's
Theorem that
$$
\coh^n(u_q^{\ge 0}(\g),\C)\cong\coh^n(u_q^{>0}(\g),\C)^{\left(\Z/\ell\Z\right)^r}
$$
for every non-negative integer $n$. By Theorem \ref{cx}, $\cx_A(\C)=\gamma
(\Ext_A^{\DOT}(\C,\C))$ for both $A=u_q^{\geq 0}(\g)$ and $A=u_q^{>0}(\g)$. We
conclude that $\cx_{u_q^{>0}(\g)}(\C)\ge\cx_{u_q^{\ge 0}(\g)}(\C)$, finishing the
proof of the assertion. 
\end{proof}

The following theorem was proved by Cibils (see the proof of \cite[Proposition~3.3]{C})
in the simply laced case for $\ell\geq 5$ by completely different methods.

\begin{thm}\label{uqn}
Let $r\ge 2$ and let $q$ be a primitive complex $\ell$th root of unity. Assume that
$\ell>1$ is an odd integer not divisible by $3$ if $\Phi$ is of type {\rm G}$_2$.
Then $u_q^{>0}(\g)$ is wild.
\end{thm}

Note that the trivial module $\C$ is the only simple $u_q^{>0}(\g)$-module (up to
isomorphism) and thus $u_q^{>0}(\g)$ has only one block. If $r=1$, then $u_q^{>0}(\g)$
is isomorphic to the truncated polynomial algebra $\C[X]/(X^\ell)$ and therefore
$u_q^{>0}(\g)$ is representation-finite.

\begin{proof}
According to Lemma \ref{cxeq} and the proof of \cite[Theorem 4.3]{FW}, we obtain
$\cx_{u_q^{>0}(\g)}(\C)=\cx_{u_q^{\ge 0}(\g)}(\C)\ge 3$ unless $r=1$. Since
$u_q^{>0}(\g)$ is a Yetter-Drinfeld Hopf algebra over the group algebra $\C
[(\Z/\ell\Z)^r]$, $u_q^{>0}(\g)$ is a Frobenius algebra (see \cite[Corollary
5.8]{FMS} or \cite[Proposition 2.10(3)]{S}) and therefore self-injective. Thus,
if $r\ge 2$, Theorem \ref{mainthm} in conjunction with Corollary \ref{fguqn}
implies that $u_q^{>0}(\g)$ is wild.
\end{proof}

\begin{remark}
A more direct proof for the fact that $u_q^{>0}(\g)$ is a Frobenius algebra can also
be given along the lines of the proof of \cite[Theorem 7.2]{BGS} for the quantum Borel
subalgebra $u_q^{\ge 0}(\g)$ by using the PBW-basis of the De Concini-Kac quantum
group associated to $\g$ \cite[Section 6.1]{BGS}. By this method one also finds that the
Nakayama automorphism $u_q^{>0}(\g)$ is trivial.
\end{remark}

Gordon \cite[Theorem 7.1(b)(i)]{G2} proved the wildness of certain finite dimensional
factor algebras of the quantized function algebra of a simply-connected connected
complex semisimple algebraic group at roots of unity of odd degree (not divisible by
$3$ if the group has a component of type G$_2$) by using Rickard's wildness criterion.
We leave it to the interested reader to fix the gap in Gordon's proof by proceeding
similarly to Section \ref{ruea}. Here \cite[Theorem 5.2]{G1} plays the role of \cite[Theorem
1.1]{FS} in order to establish an analogue of Lemma \ref{fgredenvalg} and \cite[Corollary
5.4.3]{SF} is replaced by \cite[Theorem 8.4]{BGS}. Then Theorem \ref{mainthm} and
\cite[Theorem 6.20]{G2} show that the finite dimensional quantized function algebras
in \cite[Theorem 7.1(b)(i)]{G2} are wild. But note that in \cite[Theorem 4.5(ii)]{BG}
the same result including the borderline case $l(w_1)+l(w_2)=2N-2$ is proved by using
completely different methods even without assuming that $\ell$ is larger than the Coxeter
number.


\section{Nichols algebras}\label{sec:nichols}


In this section we apply Theorem \ref{mainthm} to Nichols algebras, also known as
quantum symmetric algebras, generalizing the application to small half-quantum
groups of the previous section. We summarize some of the definitions first; for more
details, see e.g.\ \cite{AS}.

If $G$ is a finite group, a {\em Yetter-Drinfeld module\/} over $kG$ is a $G$-graded
vector space $V=\bigoplus_{g\in G}V_g$ that is also a $kG$-module for which $g\cdot
V_h = V_{ghg^{-1}}$ for all $g,h\in G$. The grading corresponds to a $kG$-comodule
structure: $\delta(v)=g\ot v$ for all $g\in G$ and $v\in V_g$. The category $\YDG$
of all Yetter-Drinfeld modules over $kG$ is a braided monoidal category with braiding
$c:U\ot V\stackrel{\sim}{\rightarrow}V\ot U$ determined by $c(u\ot v)=g(v)\ot u$ whenever
$g\in G$, $u\in U_g$, $v\in V$. A {\em braided Hopf algebra\/} in $\YDG$ is an object
$R$ of $\YDG$ having structure maps (unit, multiplication, counit, comultiplication,
coinverse) that are morphisms in the category and satisfy the usual commutative diagrams,
e.g.\ the multiplication $\mu$ is associative in the sense that $\mu\circ(\mu\ot\id_R)=
\mu\circ(\id_R\ot\mu)$. Examples of braided Hopf algebras in $\YDG$ are the tensor
algebra $T(V)$ and the {\em Nichols algebra\/} $\B(V)$ associated to a Yetter-Drinfeld
module $V$ over $kG$ (the latter is defined to be a particular quotient of the tensor
algebra $T(V)$ that can be finite dimensional).

Given a braided Hopf algebra $R$ in $\YDG$, one may form its Radford biproduct (or
its bosonization) $A:=R\# kG$: As an algebra, this is just the skew group algebra,
that is the free $R$-module with basis $G$ and multiplication $(rg)(sh)=r(g\cdot s)
gh$ for $r,s\in R$ and $g,h\in G$. It is also a coalgebra, $\Delta(rg)=\sum r^{(1)}
(r^{(2)})_{(-1)}g\ot(r^{(2)})_{(0)}g$ for all $r\in R$ and $g\in G$, where $\Delta(r)
=\sum r^{(1)}\ot r^{(2)}$ in $R$ as a Hopf algebra in $\YDG$ and $\delta(r)=\sum
r_{(-1)}\ot r_{(0)}$ is its $kG$-comodule structure. These two structures make $A$
into a Hopf algebra.

We assume that the characteristic of $k$ does not divide the order of $G$, and let
$k$ be an $R$-module via the counit map from $R$ to $k$. We will be interested in
the $G$-invariant subalgebra $\Ext^{\DOT}_{R}(k,k)^G$ of $\Ext^{\DOT}_R(k,k)$. 
Some results in \cite[Appendix]{PW} generalize to show that this subalgebra embeds
into the Hochschild cohomology $\HH^{\DOT}(R)$. We summarize these ideas here for
completeness. We will use the fact that $\Ext^{\DOT}_A(k,k)\cong \Ext^{\DOT}_R(k,k)^G$,
where $A=R\# kG$, valid since the characteristic of $k$ does not divide the order
of $G$.

First define a map $\delta:A\rightarrow A^e$ by $\delta(a)=\sum a_1\ot S(a_2)$ for
every $a\in A$. This is an injective algebra homomorphism by \cite[Lemma 11]{PW}.
Let 
$$
\D:=\bigoplus_{g\in G}(Rg\ot Rg^{-1})\,,
$$
a subalgebra of $A^e$. Note that since $R$ is in $\YDG$, the algebra $\D$ contains
the subalgebra $\delta(A)\cong A$. The proof of \cite[Lemma 11]{PW} shows that as
induced modules from $\delta(A)$ to $\D$ and to $A^e$, there are isomorphisms $\D
\ot_{\delta(A)}k\cong R$ and $A^e\ot_{\delta(A)}k\cong A$. Thus induction from
$A\cong\delta(A)$ to $\D$, and then to $A^e$, yields a sequence of isomorphisms
$$
\Ext^{\DOT}_A(k,k)\cong\Ext^{\DOT}_{\D}(R,k)\cong\Ext^{\DOT}_{A^e}(A,k)\,.
$$
The last space $\Ext^{\DOT}_{A^e}(A,k)$ embeds, as an algebra, into $\Ext^{\DOT}_{A^e}
(A,A)\cong\HH^{\DOT}(A)$ via the unit map from $k$ to $A$ (see \cite[Section 5.6]{GK}
or \cite[Lemma 12]{PW}). In fact, $\HH^{\DOT}(A)\cong\Ext^{\DOT}_{R^e}(R,A)^G\cong
\Ext^{\DOT}_\D(R,A)$, and so the second of the above isomorphisms results in an
embedding of $\Ext^{\DOT}_\D(R,k)$ into $\Ext^{\DOT}_\D(R,A)$. Its image is in
$\Ext^{\DOT}_\D(R,R)\cong \HH^{\DOT}(R)^G$, and the latter embeds into $\HH^{\DOT}(R)$. 

We will identify $\Ext^{\DOT}_R(k,k)^G$ with its image in $\HH^{\DOT}(R)$ in what
follows. Let
\begin{eqnarray*}
H^{\DOT}:=
\left\{
\begin{array}{cl}
\HH^0(R)\cdot\bigoplus\limits_{n=0}^\infty\Ext^n_R(k,k)^G\,, & \mbox{if }\cha k=2\,,\\ & \\
\HH^0(R)\cdot\bigoplus\limits_{n=0}^\infty\Ext^{2n}_R(k,k)^G\,, & \mbox{if }\cha k\neq 2\,,
\end{array}
\right.
\end{eqnarray*}
guaranteeing that $H^{\DOT}$ is a commutative algebra.

\begin{thm}\label{nichols}
Let $G$ be a finite group and let $k$ be an algebraically closed field whose
characteristic does not divide the order of $G$. Let $R$ be a finite dimensional
Hopf algebra in $\YDG$. Assume {\rm ({\bf fg1})} and {\rm ({\bf fg2})} hold for
$H^{\DOT}$ defined as above. If there is an $R$-module $M$ such that $\cx_R(M)
\geq 3$, then $R$ is wild.
\end{thm}

\begin{proof}
Since $A=R\# kG$ is a finite dimensional Hopf algebra, it is a Frobenius algebra.
Standard arguments show that $R$ is itself a Frobenius algebra: $A$ is a free
$R$-module, that is, $A\cong R^{n}$ as an $R$-module, where $n$ denotes the order
of $G$. As $A$ is Frobenius, $A\cong A^*$ as $A$-modules, so $R^n\cong(R^n)^*\cong
(R^*)^n$ as $R$-modules. By the Krull-Remak-Schmidt Theorem, $R\cong R^*$ as
$R$-modules, and so $R$ is Frobenius. (For a proof in a much more general context,
see \cite[Corollary 5.8]{FMS} or \cite[Proposition 2.10(3)]{S}.) Since $R$ is Frobenius,
it is self-injective and the result now follows from Theorem \ref{mainthm} by choosing
the block in which an indecomposable summand of $M$ of complexity at least $3$ lies.
\end{proof}

In order to apply Theorem \ref{nichols}, the following proposition may be useful
in some cases. It is the analog of Lemma \ref{cxeq} in this setting.

\begin{prop}
Let $G$ be a finite group and let $k$ be an algebraically closed field such that
the characteristic of $k$ does not divide the order of $G$. Let $R$ be a finite
dimensional Hopf algebra in $\YDG$. Then $\cx_R(M)=\cx_{R\# kG}(M)$ for any $R\#
kG$-module $M$.
\end{prop}

\begin{proof}
The skew group algebra $R\# kG$ is free as an $R$-module, so a projective resolution
of $M$ as an $R\# kG$-module restricts to a projective resolution of $M$ as an
$R$-module. Therefore $\cx_R(M)\leq\cx_{R\# kG}(M)$.
On the other hand, the $R\# kG$-module $M$ is a direct summand of the induced
module $(R\# kG)\ot_RM$: Since the characteristic of $k$ does not divide the order
$|G|$ of $G$, the canonical projection from this induced module to $M$ splits via the
map $m\mapsto\frac{1}{|G|}\sum_{g\in G}g\ot g^{-1}m$. Since a projective resolution
of $M$ as an $R$-module may be induced to a projective resolution of $(R\# kG)\ot_RM$
as an $R\# kG$-module, we now have 
$$
\cx_R(M)\geq\cx_{R\# kG}((R\# kG)\ot_RM)\geq\cx_{R\# kG}(M)\,.
$$
\end{proof}

Alternatively, we may use the following lemma combined with results from \cite{FW},
applied to the Hopf algebra $R\# kG$, to obtain a wildness criterion for Nichols
algebras. The abelian case of the lemma is \cite[Lemma 4.2]{BG}; the proof is valid
more generally using Morita equivalence of $A$ and $(A\# kG)\#(kG)^*$ (see \cite{CM}).

\begin{lem}\label{same}
Let $G$ be a finite group and let $k$ be an algebraically closed field such that the
characteristic of $k$ does not divide the order of $G$. Let $R$ be a finite dimensional
algebra with an action of $G$ by automorphisms. Then $R$ and $A=R\# kG$ have the
same representation type. 
\end{lem}

We next apply Theorem \ref{nichols} to a large class of examples from \cite{MPSW}.
These are the Nichols algebras and corresponding Hopf algebras arising in the
classification of finite dimensional pointed Hopf algebras having abelian groups
of group-like elements by Andruskiewitsch and Schneider \cite{AS}. Each such algebra
is defined in terms of the following data: Let $\theta$ be a positive integer and let
$(a_{ij})_{1\leq i,j\leq\theta}$ be a {\em Cartan matrix of finite type\/}, i.e., its Dynkin
diagram is a disjoint union of copies of some of the diagrams A$_n$, B$_n$, C$_n$,
D$_n$, E$_6$, E$_7$, E$_8$, F$_4$, G$_2$. Let $\Phi$ be the root system corresponding
to $(a_{ij})$. Let $G$ be a finite abelian group and let $\widehat{G}$ be its dual group
of characters. For each $i$, $1\leq i\leq\theta$, choose $g_i\in G$ and $\chi_i\in\widehat{G}$
such that $\chi_i(g_i)\neq 1$ and $\chi_j(g_i)\chi_i(g_j)=\chi_i(g_i)^{a_{ij}}$ for all $1\leq
i,j\leq\theta$. Let $\Da$ be the set of data $(G,(g_i)_{1\leq i\leq\theta},(\chi_i)_{1\leq i\leq
\theta},(a_{ij})_{1\leq i,j\leq\theta})$. The finite dimensional Hopf algebra $u(\Da)$ (also
denoted $u(\Da,0,0)$ in \cite{AS}) is defined to be the Radford bosonization of the Nichols
algebra $R$ corresponding to the following Yetter-Drinfeld module over $kG$: $V$ is
a vector space with basis $x_1,\ldots,x_{\theta}$. For each $g\in G$, let $V_g$ be the
$k$-linear span of all $x_i$ for which $g_i=g$, and let $g(x_i)=\chi_i(g)x_i$ for $1\leq i
\leq\theta$ and $g\in G$.

\begin{example}{\em
Let $\g$ be a complex simple Lie algebra of rank $\theta$, let $(a_{ij})$ be the
corresponding Cartan matrix, and let $\alpha_1,\ldots,\alpha_{\theta}$ be distinct
simple roots. Let $q$ be a primitive $\ell$th root of unity, $\ell$ odd and prime
to $3$ if $\g$ has a component of type G$_2$. Let $G:=(\Z/\ell\Z)^{\theta}$ with
generators $g_1,\ldots,g_{\theta}$. Let $(-,-)$ denote the positive definite symmetric
bilinear form on the real vector space spanned by the simple roots $\alpha_1,\ldots,
\alpha_{\theta}$ that is induced from the Killing form of $\g$. For each $i$, set $\chi_i
(g_j):=q^{(\alpha_i,\alpha_j)}$ and extend this to a character on $G$. The Nichols
algebra $R$ is $u_q^{>0}(\g)$ and its Radford bosonization is $R\# kG\cong u_q^{\ge
0}(\g)$.}
\end{example}

We assume that the order $N_i$ of $\chi_i(g_i)$ is odd for all $i$, and is prime to
$3$ for all $i$ in a connected component of type G$_2$.

\begin{thm}\label{pointed}
Let $R$ be the Nichols algebra and let $u(\Da)$ be the complex Hopf algebra described
above in terms of the data $\Da$ with $\theta\geq 2$. Assume $\chi_i^{N_i}$ is the
trivial character for all $i$ and that the only common solution $(c_1,\ldots,c_{\theta})
\in\{0,1\}^{\theta}$ of the equations
$$
\chi_1(g)^{c_1}\cdots\chi_{\theta}(g)^{c_{\theta}}=1
$$
for all $g\in G$ is $(c_1,\ldots,c_{\theta})=(0,\ldots,0)$. Then $R$ and $u(\Da)$ are
both wild.
\end{thm}

We note that under some conditions, $u_q^{>0}(\g)$ and $u_q^{\ge 0}(\g)$ are included
among the Nichols algebras and Hopf algebras of the form $R$ and $u(\Da)$ satisfying the
hypotheses of the theorem: For these classes of examples, $\chi_i(g_i)=q^{(\alpha_i,\alpha_i)}$,
whose order $N_i$ is just the order $\ell$ of $q$. Thus $\chi_i^{N_i}$ is indeed the trivial
character for all $i$. A common solution $(c_1,\ldots,c_{\theta})$ of the equations $\chi_1
(g)^{c_1}\cdots\chi_{\theta}(g)^{c_{\theta}}=1$ for all $g\in G$ is equivalent to a common
solution for the generators $g=g_j$ for all $j$, which is a solution to the equations
$$
(\alpha_1,\alpha_j)c_1+\cdots+(\alpha_{\theta},\alpha_j)c_{\theta}\equiv 0\,\,(\mbox{mod }\ell)
$$
for all $j$. Ginzburg and Kumar \cite[Section 2.5, pp.\ 187/88]{GK} showed that under
their hypotheses, no non-zero solution exists in $\{0,1\}^{\theta}$: They assume that
the order $\ell$ of the root of unity $q$ is odd, greater than the Coxeter number of
$\g$, and prime to $3$ if $\g$ has a component of type G$_2$.
This provides an alternative approach to the results in Section \ref{small}.

\begin{proof}[Proof of Theorem \ref{pointed}]
Let
\begin{eqnarray*}
H^{\DOT}:=
\left\{
\begin{array}{cl}
\HH^0(R)\cdot\bigoplus\limits_{n=0}^\infty\Ext^n_{u(\Da)}(k,k)\,, & \mbox{if }\cha k=2\,,\\ & \\
\HH^0(R)\cdot\bigoplus\limits_{n=0}^\infty\Ext^{2n}_{u(\Da)}(k,k)\,, & \mbox{if }\cha k\neq 2\,.
\end{array}
\right.
\end{eqnarray*}
By \cite[Corollary 5.5]{MPSW}, $\Ext^{\DOT}_{u(\Da)}(k,k)$ is isomorphic to the $G$-invariant
subalgebra of a polynomial algebra in $|\Phi^+|\geq 3$ indeterminates (in cohomological
degree $2$), and the group $G$ acts diagonally on these indeterminates. Therefore assumption
({\bf fg1}) holds for $R$. 
Similarly, for $u(\Da)$, we take $H^{\DOT}$ as for $R$ but replace the
factor $\HH^0(R)$ with $\HH^0(u(\Da))$, and see that ({\bf fg1}) holds
for $u(\Da)$. 

We claim that by Remark \ref{TFAE}, ({\bf fg2}) holds for $R$: By \cite[Theorem
5.3]{MPSW}, ({\bf fg2}) holds for the corresponding Hopf algebra $u({\mathcal D})\cong R\# kG$,
since $\Ext^{\DOT}_{u({\mathcal D})}(U,V)\cong\Ext^{\DOT}_{u({\mathcal D})}(k,U^*\ot V)$ for
all finite dimensional $u({\mathcal D})$-modules $U$ and $V$. Next, let $N$ be a finite
dimensional $R$-module. By the Eckmann-Shapiro Lemma,
$$
\Ext^{\DOT}_{R\# kG}((R\# kG)\ot_R M,(R\# kG)\ot_R N)\cong\Ext^{\DOT}_R(M,(R\# kG)\ot_R N)\,.
$$
Now $N$ is a direct summand of $(R\# kG)\ot_R N$ as an $R$-module, and therefore
$\Ext^{\DOT}_R(M,N)$ is finitely generated over $\Ext^{\DOT}_{R\# kG}(k,k)\cong\Ext^{\DOT}_R
(k,k)^G$ since $\Ext^{\DOT}_R(M,(R\# kG)\ot_R N)$ is. Under our assumptions, $\cx_R(k)\geq 3$,
so by Theorem \ref{nichols}, $R$ is wild. Replacing $R$ by $u(\Da)$, we obtain the
analogous conclusion for $u(\Da)$.
\end{proof}

The hypothesis regarding solutions $(c_1,\ldots,c_{\theta})$ in the theorem is surely not
necessary. However, in general, it is more difficult to determine precisely the Krull
dimension of the relevant $G$-invariant cohomology ring.


\section{Comparison of varieties for Hopf algebras}


In this section, we look at finite dimensional Hopf algebras in general, explaining
the connection between the varieties defined in terms of Hochschild cohomology
used in this article, and those defined in terms of the cohomology of the trivial module
used in \cite{FW}.

Let $A$ be a finite dimensional Hopf algebra over $k$. Then $A$ is a Frobenius algebra,
and therefore self-injective. Let $\coh^{\DOT}(A,k):=\Ext^{\DOT}_A(k,k)$, and let
\begin{eqnarray*}
\coh^{\ev}(A,k):=
\left\{
\begin{array}{cl}
\bigoplus\limits_{n=0}^\infty\coh^n(A,k)\,, & \mbox{if }\cha k=2\,,\\ & \\
\bigoplus\limits_{n=0}^\infty\coh^{2n}(A,k)\,, & \mbox{if }\cha k\neq 2\,.
\end{array}
\right.
\end{eqnarray*}

We will need the following assumption, as in \cite{FW}.
\vspace{.2cm}

\noindent Assumption ({\bf fg}):
\vspace{-.2cm}
\begin{equation*}
\begin{array}{l}
\mbox{{\em Assume }}\coh^{\ev}(A,k) \ \mbox{{\em is finitely generated, and that
for any two finite dimensional}}\\
A\mbox{{\em -modules }}M\mbox{ {\em and} }N, \mbox{ {\em the }} \coh^{\ev}(A,k)
\mbox{{\em -module }}\Ext^{\DOT}_A(M,N) \ \mbox{{\em is  finitely generated.}}
\end{array}
\end{equation*}
In \cite{FW}, the cohomology ring $\coh^{\ev}(A,k)$ was used to define varieties
for modules: For each $A$-module $M$, $\V_A(M)$ is the maximal ideal spectrum of
$\coh^{\ev}(A,k)/I_A(M)$, where $I_A(M)$ is the annihilator of the action of
$\coh^{\ev}(A,k)$ on $\Ext^{\DOT}_A(M,M)$.

As mentioned in Section \ref{sec:nichols}, the cohomology ring $\coh^{\DOT}(A,k)$
embeds into Hochschild cohomology $\HH^{\DOT}(A)$ as a graded subalgebra (see
\cite[Section 5.6]{GK} or \cite[Lemma~12]{PW}). Let $B_0$ be the principal block of
$A$. Note that $\coh^{\DOT}(A,k)$ in fact embeds into $\HH^{\DOT}(B_0)$. Under
assumption ({\bf fg}), let
$$
H^{\DOT}:=\HH^0(B_0) \cdot \coh^{\ev}(A,k)\,,
$$
that is, $H^{\DOT}$ is the subalgebra of $\HH^{\DOT}(B_0)$ generated by $\HH^0(B_0)$
and the image of $\coh^{\ev}(A,k)$ in $\HH^{\DOT}(B_0)$. Then $H^{\DOT}$ satisfies
({\bf fg1}) and ({\bf fg2}). By Theorem \ref{cx}, Lemma~\ref{cxblock}, and \cite[Proposition
2.3]{FW}, for any finite dimensional $B_0$-module $M$,
\begin{equation}\label{HA}
\dim\V_H(M)=\cx_{B_0}(M)=\cx_A(M)=\dim\V_A(M)\,.
\end{equation}
That is, $\dim\V_H(M)=\dim\V_A(M)$, as one expects, since $\HH^0(B_0)\cong Z(B_0)$ is
finite dimensional. In fact, the varieties $\V_H(M)$ and $\V_A(M)$ are isomorphic. We
explain this next.

By \cite[Lemma 13]{PW}, the following diagram commutes for every finite dimensional
$A$-module $M$:
\begin{equation*}
\begin{xy}*!C\xybox{
\xymatrix{
   \coh^{\DOT}(A,k) 
     \ar[drr]^{-\ot_k M}\ar[d]_{A^e\ot_A -}\\
  \HH^{\DOT}(A) \ar[rr]^{-\ot_A M \hspace{.1cm}}
   &&\Ext^{\DOT}_A(M,M)
}}\end{xy}
\end{equation*}
The vertical arrow is an embedding that is induced by an embedding from $A$ to $A^e$ (see
e.g.\ \cite[Lemma 11 and Lemma 12]{PW}). This diagram shows that the action of $\coh^{\ev}(A,k)$
on $\Ext^{\DOT}_A(M,M)$ factors through $\HH^{\DOT}(A)$, resulting in a commutative diagram
\begin{equation*}
\begin{xy}*!C\xybox{
\xymatrix{
   \coh^{\ev}(A,k)
     \ar[drr]^{-\ot_k M} \ar[d]\\
    H^{\DOT} \ar[rr]^{-\ot_A M \hspace{.1cm}}
    && \Ext^{\DOT}_A(M,M)
}}\end{xy}
\end{equation*}
Again, the vertical arrow is an embedding and induces an embedding
$$
\coh^{\ev}(A,k)/I_A(M)\longrightarrow H^{\DOT}/I_H(M)\,.
$$
As $\HH^0(B_0)\cong Z(B_0)$ is a local algebra, the above embedding induces an isomorphism
modulo radicals. Thus $\V_H(M)$ and $\V_A(M)$ are isomorphic.

Applying Theorem \ref{mainthm} and Corollary \ref{maincor}, in light of (\ref{HA}), we
recover versions of \cite[Theorem 3.1 and Corollary 3.2]{FW} for the principal block.
These results were originally proven using Hopf-theoretic techniques. In particular,
tensor products of modules were used in \cite[Corollary 2.6]{FW}, crucial in the proofs
of \cite[Theorem~3.1 and Corollary 3.2]{FW}.

\begin{remark}
In case $G$ is a finite group and $A=kG$, much more is true: Linckelmann \cite{Li1} proved
that one may take $H^{\DOT}$ instead to be the full even Hochschild cohomology ring $\HH^{\ev}
(B_0)$ and still obtain a isomorphism between the varieties $\V_H(M)$ and $\V_A(M)$. More
generally, Linckelmann defined block cohomology for non-principal blocks in this setting and
proved that the resulting variety of a module is isomorphic to that obtained via Hochschild
cohomology of the block.
\end{remark}



\end{document}